\documentclass[12pt]{amsart}    
\usepackage{latexsym,mathtools}
\usepackage{epsfig,setspace}
\usepackage{a4}
\usepackage{amssymb}
\usepackage{amsthm}
\usepackage{amsbsy}
\usepackage{amsfonts,amssymb,latexsym,amscd,amsmath,euscript,enumerate,verbatim,calc}
\allowdisplaybreaks[1]
\usepackage{url}
\usepackage{hhline}
\usepackage{pifont}
\usepackage[colorlinks=true,linkcolor=blue,citecolor=magenta]{hyperref}

\theoremstyle{plain}

\newtheorem{theorem}{Theorem}[section]
\newtheorem{lemma}[theorem]{Lemma}
\newtheorem{proposition}[theorem]{Proposition}
\newtheorem{question}[theorem]{Question}
\newtheorem{conjecture}[theorem]{Conjecture}
\newtheorem{corollary}[theorem]{Corollary}
\theoremstyle{definition}

\newtheorem{remark}[theorem]{Remark}

\def\F{{\mathbb F}}

\def\Fq{{\mathbb F}_q}
\def\Fqm{{\mathbb F}_{q^m}}
\def\Fqt{{\mathbb F}_{q^t}}

\def\Fqmn{{\mathbb F}_{q^{mn}}}

\def\sqmn{S_q(m,n)}
\def\nqmn{N_q(m,n)}
\def\dmnq{\Delta_q(m,n)}
\def\tsr{{\rm TSR}}
\def\tsri{{\rm TSRI}}
\def\tsrp{{\rm TSRP}}

\newcommand{\GL}{\operatorname{GL}}

\newcommand{\M}{\operatorname{M}}

\newcommand{\pmn}{{\EuScript P}(mn;q)}
\newcommand{\imn}{{\EuScript I}(mn;q)}

\newcommand{\tsrmnq}{{\rm TSR}(m,n;q)}
\newcommand{\ptsrmnq}{{\rm TSR^*}(m,n;q)}
\newcommand{\tsrimnq}{{\rm TSRI}(m,n;q)}
\newcommand{\tsrpmnq}{{\rm TSRP}(m,n;q)}
\newcommand{\I}[2]{{\EuScript I}(#1;#2)}

\def\phitm{\phi_{T_{(m)}}}

\makeatletter
\def\imod#1{\allowbreak\mkern10mu({\operator@font mod}\,\,#1)}
\makeatother

\title{Enumeration of Linear Transformation Shift Registers}

\author{Samrith Ram}
\address{Institut de Math\'{e}matiques de Luminy \newline \indent
Luminy Case 907 \newline \indent
13288 Marseille Cedex 9 \newline \indent
France \newline \indent}
\email{samrith@gmail.com}



\begin{document}
%replace1
\begin{abstract}
We consider the problem of counting the number of linear transformation shift registers (TSRs) of a given order over a finite field. We derive explicit formulae for the number of irreducible TSRs of order two. An interesting connection between TSRs and self-reciprocal polynomials is outlined. We use this connection and our results on TSRs to deduce a theorem of Carlitz on the number of self-reciprocal irreducible monic polynomials of a given degree over a finite field.    
\end{abstract}
%replace1%
\maketitle
%replace2
\section{Introduction}
A linear feedback shift register (LFSR) is a mechanism for generating a sequence in a finite field. LFSRs have a plethora of practical applications and are frequently used in generating pseudorandom numbers, fast digital counters and stream ciphers. A generalization of LFSR called word-oriented feedback shift register ($\sigma$-LFSR) was considered by Zeng, Han and He \cite{Zeng}. For LFSRs as well as $\sigma$-LFSRs, those that are primitive (i.e., for which the corresponding infinite sequence is of maximal possible period) are of particular interest. The following conjecture was proposed in the binary case in \cite{Zeng} and was extended to the $q$-ary case in \cite{GSM}:
\begin{conjecture}
\label{zhhconj}
For positive integers $m$ and $n$, the number of primitive $\sigma$-LFSRs of order $n$ over $\Fqm$ is given by
\begin{equation}
\label{noprimlfsr}
\frac{\phi(q^{mn}-1)}{mn} \, q^{m(m-1)(n-1)} \prod_{i=1}^{m-1}(q^m-q^i). 
\end{equation}
\end{conjecture}
The notion of $\sigma$-LFSR is essentially equivalent to that of a splitting subspace previously defined by Niederreiter \cite{N2}: Given positive integers $m,n$ and $\alpha \in \Fqmn$, an $m$-dimensional $\Fq$-linear subspace of $\Fqmn$ is said to be $\alpha$-splitting if 
$$
\Fqmn=W\oplus \alpha W\oplus \cdots \oplus \alpha^{n-1} W.
$$
Splitting subspaces were studied by Niederreiter \cite{N2} in the context of his work on the multiple recursive matrix method for pseudorandom number generation. In his paper \cite[p. 11]{N2}, he asked the following question, stating it was an open problem: If $\alpha$ generates the cyclic group $\Fqmn^*$, what is the number of $m$-dimensional $\alpha$-splitting subspaces? More generally, we may ask: 
\begin{question}
\label{howmanysplit}
Given $\alpha \in \Fqmn$ such that $\Fqmn=\Fq(\alpha)$, what is the number of $m$-dimensional $\alpha$-splitting subspaces? 
\end{question}
It was shown in \cite{m=2} that the problem of enumeration of splitting subspaces is equivalent to counting certain block companion matrices which turn out to be the state transition matrices of $\sigma$-LFSRs. We refer to Ghorpade, Hasan and Kumari~\cite{GSM}, Ghorpade and Ram \cite{m=2,split} and Chen and Tseng \cite{sscffa} for recent progress on the above question. In particular, the work of Chen and Tseng answers Question \ref{howmanysplit} completely by proving the Splitting Subspace Conjecture \cite[Conj. 5.5]{m=2}. The Splitting Subspace Conjecture (SSC) proves the Primitive Fiber Conjecture \cite[Conj. 2.3]{m=2} which in turn settles Conjecture \ref{zhhconj} in the affirmative. The SSC also establishes the Irreducible Fiber Conjecture \cite[Conj. 2.4]{m=2} which leads to a formula, similar to \eqref{noprimlfsr}, for the number of irreducible $\sigma$-LFSR of order $n$ over $\Fqm$:
$$
\left(\sum_{d\mid mn}\mu(d)q^{mn/d}\right) q^{m(m-1)(n-1)} \prod_{i=1}^{m-1}(q^m-q^i). 
$$

A subcategory of $\sigma$-LFSRs called transformation shift registers (TSRs) was considered by Tsaban and Vishne \cite{TV} to solve a problem of Preneel \cite{Preneel}. We refer to the papers of Dewar and Panario \cite{MdDp,mutualirred} for subsequent developments on TSRs. It turns out that the TSRs have very good cryptographic properties when the corresponding characteristic polynomial is primitive. Tsaban and Vishne noted in \cite{TV} that irreducible TSRs contain a high proportion of primitive TSRs. This motivates the study of irreducible TSRs in Section \ref{irreducibletsr}.

While $\sigma$-LFSRs have been studied in great detail, very little is known about the TSRs; indeed, given positive integers $m,n$ and a prime power $q$, it is not even known if there exists an irreducible TSR of order $n$ over $\Fqm$. A more difficult problem would be to determine the number of primitive or irreducible TSRs.  

In this paper, we adopt a matrix theoretic approach to enumerating TSRs by working with their state transition matrices. We are mainly interested in the
number of irreducible TSRs of a given order. We derive a recurrence which leads
to a formula \eqref{eq:n=2} for the number of irreducible TSRs of order two
over $\Fqm$ for arbitrary $m$. We then outline a connection between
irreducible TSRs and self reciprocal polynomials and give a simple method to
construct such TSRs from self-reciprocal polynomials. The results on TSRs are
used to deduce a theorem of Carlitz \cite{Car} on the number of self-reciprocal irreducible monic polynomials of a given degree over a finite field. Finally, we obtain bounds on the number of
primitive TSRs and the number of irreducible TSRs.  

\section{Preliminaries}
Throughout this paper, $m$ and $n$ are positive integers and $q$ is a prime power. We define a $(m,n)$-\emph{\tsr} matrix over $\Fq$ to be a matrix $T\in \M_{mn}(\Fq)$ of the form

%shall mean a $mn \times mn$ matrix $T$ of the following block form:
\begin{equation}
\label{tsr}
T =
\begin {pmatrix}
\mathbf{0} & \mathbf{0} & \mathbf{0} & . & . & \mathbf{0} & \mathbf{0} &c_0 B\\
I_m & \mathbf{0} & \mathbf{0} & . & . & \mathbf{0} & \mathbf{0} & c_1 B\\
. & . & . & . & . & . & . & .\\
. & . & . & . & . & . & . & .\\
\mathbf{0} & \mathbf{0} & \mathbf{0} & . & . & I_m & \mathbf{0} & c_{n-2} B\\
\mathbf{0} & \mathbf{0} & \mathbf{0} & . & . & \mathbf{0} & I_m & c_{n-1} B
\end {pmatrix}, %_{mn \times mn}
\end{equation}
where $ c_0, \dots , c_{n-1}\in \Fq$, $B\in \M_m(\Fq)$ and $I_m$ denotes the $m\times m$ identity matrix over $\Fq$, while $\mathbf{0}$ indicates the zero matrix in $\M_m(\Fq)$. We denote by $\tsrmnq$ the set of all $(m,n)$-TSR matrices over $\Fq$. Matrices in $\tsrmnq$ are precisely the state transition matrices \cite[\S 4]{SPW} of TSRs of order $n$ over $\Fqm$.\footnote{By `order' of a TSR, we mean the order of the recurrence relation defining the TSR, not the multiplicative order of the corresponding state transition matrix (if and when it lies in $\GL_{mn}(\Fq)$). We follow this convention throughout.} We often identify TSRs with their corresponding state transition matrices by referring to `TSR matrix' as simply `TSR'. The map 
$$
\Phi :\M_{mn}(\Fq) \to \Fq[X] \quad \mbox{ defined by } \quad \Phi(T): = \det\left(XI_{mn} - T \right)
$$
will be referred to as the \emph{characteristic map}. The restriction of $\Phi$ to $\tsrmnq$ will be denoted by $\Phi_{(m,n)}$.  The matrices in $\tsrmnq$ which have a primitive characteristic polynomial over $\Fq$ are denoted by $\tsrpmnq$ and those with irreducible characteristic polynomial are denoted by $\tsrimnq$. For each positive integer $r$, we denote by ${\EuScript I}(r;q)$ and ${\EuScript P}(r;q)$ the set of monic irreducible polynomials in $\Fq[X]$ of degree $r$ and the set of primitive polynomials in $\Fq[X]$ of degree $r$ respectively. Thus $\Phi$ maps $\tsrimnq$ into $\imn$ and $\tsrpmnq$ into $\pmn$. The restrictions of $\Phi$ yield the following maps:
\begin{equation}
\label{restrictions}
\Phi_P: \tsrpmnq \to \pmn \quad \mbox{and} \quad \Phi_I: \tsrimnq\to \imn.
\end{equation}

We denote the intersection $\tsrmnq \cap \GL_{mn}(\Fq)$ by $\ptsrmnq$. Elements of $\ptsrmnq$ are precisely \cite[Prop. 4]{SPW} the state transition matrices of periodic TSRs. Alternatively, $\ptsrmnq$ consists of precisely those matrices in $\tsrmnq$ whose characteristic polynomial does not vanish at zero. It follows easily from \eqref{tsr} that $T\in \ptsrmnq$ if and only if $T$ is of the form 
$$
\begin {pmatrix}
\mathbf{0} & \mathbf{0} & \mathbf{0} & . & . & \mathbf{0} & \mathbf{0} & A\\
I_m & \mathbf{0} & \mathbf{0} & . & . & \mathbf{0} & \mathbf{0} & c_1 A\\
. & . & . & . & . & . & . & .\\
. & . & . & . & . & . & . & .\\
\mathbf{0} & \mathbf{0} & \mathbf{0} & . & . & I_m & \mathbf{0} & c_{n-2} A\\
\mathbf{0} & \mathbf{0} & \mathbf{0} & . & . & \mathbf{0} & I_m & c_{n-1} A
\end {pmatrix}
$$
where $A\in \GL_m(\Fq)$ and $c_i \in \Fq$ for $1 \leq i \leq n-1$. The characteristic polynomial of $T$ is given by (\cite[Lemma 1]{SPW})
\begin{equation}
\label{charTSR}
\Phi(T)=\det(X^n I_n-g_T(X)T_{(m)}),
\end{equation}
where $g_T(X)=1+c_1 X+\ldots+c_{n-1}X^{n-1}\in \Fq[X]$ and $T_{(m)}$ denotes the submatrix of $T$ formed by the first $m$ rows and last $m$ columns of $T$. Note that $T$ is uniquely determined by $g_T(X)$ and $T_{(m)}$. It is easy to see that $\tsrimnq \subseteq \ptsrmnq$ for $\max\{m,n\}>1$. In what follows, we always assume $\max\{m,n\}>1$ unless otherwise stated. 

For every matrix $M$ we denote by $\phi_M(X)$ the characteristic polynomial of $M$. It follows from \eqref{charTSR} that for $T \in \ptsrmnq$
\begin{equation}
\label{phitphia}
\phi_T(X)=g_T(X)^m \phi_{T_{(m)}}\left(\frac{X^n}{g_T(X)}\right)
\end{equation}

Thus if $\phi_T(X)$ is irreducible in $\Fq[X]$, then so is $\phi_{T_{(m)}}(X)$. However, the converse is not true in general. For example if $g_T(X)=1$, then 
$$
\phi_T(X)=\phitm(X^n)
$$ 
which is not irreducible when $n$ is a multiple of $q$. If $\phi_T(X)$ is primitive in $\Fq[X]$, then it is not necessarily true that $\phitm(X)$ is primitive. Consider $T\in \tsr(1,2;3)$ given by 
\begin{equation*}
T=
\begin{pmatrix}
0 & 1 \\
1 & 1 
\end{pmatrix}
\end{equation*}
In this case $\phi_T(X)=X^2-X-1$ is primitive but $\phitm(X)=X-1$ is not.

The next proposition describes the form of $\phitm(X)$ when $T\in \tsrpmnq$. First, we need a lemma.

\begin{lemma}
\label{f0}
If $N$ is a positive integer and $f(X)\in {\EuScript P}(N;q)$ then $(-1)^N f(0)$ is a primitive element of $\Fq$.
\end{lemma}
\begin{proof}
See \cite[Thm. 3.18]{LN}.
\end{proof}
\begin{proposition}
\label{primblock}
If $T\in \tsrpmnq$ then 
$$(-1)^{m(n+1)}\phitm\left((-1)^{n+1}X\right)\in {\EuScript P}(m;q).$$
\end{proposition}
\begin{proof}
Let $\alpha_i (1 \leq i \leq m)$ be the roots of $\phitm$ (which is necessarily irreducible in $\Fq[X]$) in $\Fqm$. Then
$$
\phi_T(X)=\prod_{i=1}^{m}(X^n-\alpha_i g_T(X))
$$ 
is a factorization of $\phi_T(X)$ into irreducible polynomials in $\Fqm[X]$. Then, for each $i$, $X^n-\alpha_i g_T(X)$ is necessarily primitive in $\Fqm[X]$. By Lemma \ref{f0}, $(-1)^{n+1}\alpha_i$ is primitive in $\Fqm$ for each $i$. It is easily seen that the $m$ elements $(-1)^{n+1}\alpha_i$ are also conjugates of each other over $\Fq$. Equivalently, 
$$
\prod_{i=1}^{m}(X+(-1)^n\alpha_i) \in {\EuScript P}(m;q). 
$$
This is equivalent to the statement of the proposition.
\end{proof}

\begin{corollary}
 If $char(\Fq)=2$ and $\phi_T(X)$ is primitive, then so is $\phitm(X)$.
\end{corollary}

\begin{corollary}
 If $n$ is odd and $\phi_T(X)$ is primitive, then so is $\phitm(X)$.
\end{corollary}

% \begin{corollary}
% If $T\in \tsrpmnq$ then the number of possible values of $\phitm(X)$ is at most $\frac{\varphi(q^m-1)}{m}.$
% where $\varphi$ denotes the Euler Totient function. 
% \end{corollary}
% \begin{proof}
%  This follows easily since $ |{\EuScript P}(m;q)|=\frac{\varphi(q^m-1)}{m}.$
% \end{proof}

\section{Fibers of the Characteristic Map}
The maps $\Phi_I$ and $\Phi_P$ defined in \eqref{restrictions} are not surjective in general. To see this, let $T\in \tsr(2,2;2)$. We show that the primitive polynomial $X^4+X+1\in\F_2[X]$ cannot be the characteristic polynomial of $T$. Suppose, to the contrary, that 
$$
\phi_T(X)=X^4+X+1.
$$
Let $\phitm(X)=X^2+aX+b$. Then  
$$
X^4+aX^2g_T(X)+bg_T(X)^2=X^4+X+1.
$$
Formally differentiating with respect to $X$ on both sides, we obtain
$$
aX^2{g_T}'(X)=1
$$
which is impossible.

Since $\Phi_I$ is not surjective in general, the following natural question arises.

\begin{question}
\label{whatiscard}
Which polynomials $f(X)\in \Fq[X]$ are the characteristic polynomial of some $T\in \tsrimnq$ and what is the cardinality of the fiber $\Phi_{(m,n)}^{-1}(f(X))$.  
\end{question}

It follows easily from \eqref{phitphia} that $f(X)\in \Phi(\ptsrmnq)$ if and only if $f(X)$ can be expressed in the form 
\begin{equation}
\label{splits}
g(X)^m h\left(\frac{X^n}{g(X)}\right)
\end{equation}
for some monic polynomial $h(X)\in\Fq[X]$ of degree $m$ with $h(0)\neq0$ and a not necessarily monic $g(X)\in\Fq[X]$ of degree at most $n-1$ with $g(0)=1$. 

We say that a polynomial $f(X)\in\Fq[X]$ is \emph{$(m,n)$-decomposable} if it is the characteristic polynomial of some matrix in $\ptsrmnq$. We refer to \eqref{splits} as an $(m,n)$-decomposition of $f(X)$. We further say that $f(X)$ is \emph{uniquely} \emph{$(m,n)$-decomposable} if the representation of $f$ in the form \eqref{splits} is unique.  

The following theorem will be used to provide a partial answer to Question \ref{whatiscard}.
\begin{theorem}
\label{gerstenhaber}
Let $f(X)\in\Fq[X]$ be a monic polynomial of degree $n$ and let $f=f_1^{m_1} \cdots f_k^{m_k}$, where the $f_i$ are distinct irreducible polynomials in $\Fq[X]$ of degree $d_i$. The number of matrices in $\M_n(\Fq)$ that have $f(X)$ as their characteristic polynomial is given by 
\[ N_\chi(f(X))=q^{n^2-n}  \frac{F(q,n)}{\prod_{i=1}^{k}{F(q^{d_i},m_i)}}. \]
where \[ F(q,r)=\prod_{i=1}^{r}(1-q^{-i}) \] for every positive integer $r$.

\end{theorem}
\begin{proof}
See \cite[\S 2]{Mg} or \cite[Thm. 2]{Ir}. 
\end{proof}

\begin{theorem}
\label{uniqsplit}
Suppose $f(X)$ is uniquely $(m,n)$-decomposable as
$$g(X)^m h\left(\frac{X^n}{g(X)}\right).$$ 
Then,
$$
|\Phi^{-1}_{(m,n)}(f(X))|=N_{\chi}(h(X)).
$$ 
\end{theorem}
\begin{proof}
Suppose $T\in \ptsrmnq$ and $\phi_T(X)=f(X)$. By the hypothesis, $g_T(X)$ and $\phitm(X)$ are uniquely determined and are equal to $g(X)$ and $h(X)$ respectively. Thus the number of such $T$ is equal to the number of possible values of $T_{(m)}$ with $\phitm(X)=h(X)$. This is the statement of the theorem.
\end{proof}

\begin{corollary}
\label{fiberofphit}
Suppose $T\in \ptsrmnq$ and $\phi_T(X)$ is uniquely $(m,n)$-decomposable. Then  
$$
|\Phi^{-1}_{(m,n)}(\phi_T(X))|=N_{\chi}(\phitm(X)).
$$
\end{corollary}

\begin{theorem}
\label{irrisdecomp}
Suppose $f(X)$ is $(m,n)$-decomposable and irreducible in $\Fq[X]$. Then $f(X)$ is uniquely $(m,n)$-decomposable. 
\end{theorem}
\begin{proof}
Let 
$$
f(X)=g_1(X)^m h_1\left(\frac{X^n}{g_1(X)}\right)=g_2(X)^m h_2\left(\frac{X^n}{g_2(X)}\right)
$$ 
be two $(m,n)$-decompositions of $f(X)$. Since $f$ is irreducible, so are $h_1$ and $h_2$. Let 
$$
h_1(X)=\prod_{i=1}^{m}\left(X-\lambda_i\right) \quad \mbox{ and } \quad h_2(X)=\prod_{i=1}^{m}\left(X-\mu_i\right)
$$
be the factorizations of $h_1$ and $h_2$ in $\Fqm[X]$. Then 
$$
\prod_{i=1}^{m}\left(X^n-\lambda_i g_1(X)\right) \quad \mbox{ and } \quad \prod_{i=1}^{m}\left(X^n-\mu_i g_2(X)\right)
$$
are two factorizations of $f(X)$ in $\Fqm[X]$. Since $f(X)$ is irreducible of degree $mn$, $f(X)$ splits uniquely into $m$ distinct irreducible factors of degree $n$ in $\Fqm[X]$. Thus each factor in both the above products is irreducible and the factors in one product are merely a rearrangement of those in the other. Thus there exists a permutation $\sigma\in \mathfrak{S}_m$ such that 
$$
X^n-\lambda_i g_1(X)=X^n-\mu_{\sigma(i)} g_2(X) \quad \mbox{ for } 1\leq i \leq m.
$$
Since $g_1(0)=g_2(0)$, it follows that $\lambda_i =\mu_{\sigma(i)}$ for $1\leq i \leq m$ and hence $g_1(X)=g_2(X)$.  Since the $\lambda_i$  are a permutation of the $\mu_j$ it follows that $h_1(X)=h_2(X)$ as well, proving uniqueness.
\end{proof}

\begin{theorem}
\label{fibersize}
If $T\in \tsrimnq$ then 
$$
\left|\Phi_{(m,n)}^{-1}\left(  \phi_T(X) \right)\right|=\frac{|\GL_m(\Fq)|}{q^m-1}.
$$
\end{theorem}
\begin{proof}
If $T$ is as above then $\phi_T(X)$ is irreducible and $(m,n)$-decomposable. By Theorem \ref{irrisdecomp} $\phi_T(X)$ is uniquely $(m,n)$-decomposable and Corollary \ref{fiberofphit} yields
$$
\left|\Phi_{(m,n)}^{-1}\left(  \phi_T(X) \right)\right|=N_{\chi}(\phitm(X)).
$$  
Since $\phitm(X)$ is also irreducible it follows from Theorem \ref{gerstenhaber} that
$$
N_{\chi}(\phitm(X))= \frac{|\GL_m(\Fq)|}{q^m-1}.
$$
\end{proof}

\section{TSRs with an Irreducible Characteristic Polynomial}
\label{irreducibletsr}
We now compute the number of irreducible TSRs in some simple cases.
\begin{theorem}
\begin{align*}
|\tsri(1,n;q)| & = |{\EuScript I}(n;q)| \\
|\tsri(m,1;q)| & = \frac{|\GL_m(\Fq)|}{q^m-1} |{\EuScript I}(m;q)|\\
|\tsrp(1,n;q)| & = |{\EuScript P}(n;q)| \\
|\tsrp(m,1;q)| & = \frac{|\GL_m(\Fq)|}{q^m-1} |{\EuScript P}(m;q)|
\end{align*}
\end{theorem}
\begin{proof}
If either $m$ or $n$ equals 1, it is easily seen that the maps $\Phi_I$ and $\Phi_P$ are surjective. The above formulae follow easily from Theorem \ref{fibersize}. 
\end{proof}

Let $\sqmn$ denote the set of irreducible polynomials $f(X)\in\Fqm[X]$ of the form 
$$
X^n-\lambda g(X) 
$$ 
where $\lambda$ satisfies $\Fqm=\Fq(\lambda)$ and $g(X)\in\Fq[X]$ with $g(0)=1$ and $\deg g(X)\leq n-1$. The significance of $\sqmn$ is apparent from the following theorem.
\begin{theorem}
\label{tsrisqmn}
For positive integers $m,n$
$$
|\tsrimnq|=\frac{|\sqmn|}{m} \frac{|\GL_m(\Fq)|}{q^m-1}.
$$
\end{theorem}
\begin{proof}
Define $$\dmnq:=\Phi_I\left(\tsrimnq\right).$$ By Theorem \ref{fibersize},
\begin{equation}
\label{deltasize}
|\tsrimnq|= |\dmnq|\frac{|\GL_m(\Fq)|}{q^m-1}.
\end{equation}
Define a map 
$$\Gamma : \sqmn \to \Fqmn[X]$$ 
by
$$\Gamma(\left(X^n-\lambda g(X)\right) := \prod_{i=0}^{m-1}{\left(X^n-\lambda^{q^i}g(X)\right)}.$$
It is easy to see that the product on the right is $(m,n)$-decomposable. Let $\beta$ be a root of $X^n-\lambda g(X)$ in some extension field of $\Fqm$. Then, the minimal  polynomial of $\beta$ over $\Fq$ is clearly $\Gamma(X^n-\lambda g(X))$. Thus $\Gamma(X^n-\lambda g(X))$ is irreducible in $\Fq[X]$.  Since $\dmnq$ is precisely the set of irreducible $(m,n)$-decomposable polynomials in $\Fq[X]$, it follows that 
$
\Gamma\left(\sqmn\right)\subseteq \dmnq.
$
We claim that
$$
\Gamma\left(\sqmn\right)= \dmnq.
$$
To see this, let $f(X)\in \dmnq$. Since $f$ is irreducible, $f$ has a unique $(m,n)$-decomposition, say
\begin{equation*}
\label{decomp}
f(X)=g(X)^m h\left(\frac{X^n}{g(X)}\right).
\end{equation*}
Then $h(X)$ is necessarily irreducible in $\Fq[X]$ and if $\mu$ is a root of $h(X)$ in $\Fqm$, then 
$$
\Gamma(X^n-\mu g(X))=f(X),
$$ 
proving the claim. It is now easy to see that $\Gamma^{-1}(f(X))$ is precisely the set $\{X^n-\mu^{q^i}g(X):0 \leq i \leq m-1 \}$. Thus $|\Gamma^{-1}(f)|=m$ for each $f\in \dmnq$ and consequently 
$$
|\dmnq|=\frac{|\sqmn|}{m} .
$$ 
The theorem now follows from \eqref{deltasize}. 
\end{proof}

\section{Irreducible TSRs of order two}

In this section we outline a connection between TSRs of order two ($n=2$) and self-reciprocal polynomials and give a new proof of a theorem of Carlitz \cite{Car} (which has been reproved by Ahmadi~\cite{Ahmadi}, Cohen \cite{Sc}, Meyn \cite{Meyn}, Meyn and G{\"o}tz \cite{HmWg} and Miller \cite{Miller}) on the number of self-reciprocal irreducible monic polynomials of a given degree over a finite field. In what follows, we denote the cardinality of $\sqmn$ (as defined in Section \ref{irreducibletsr}) by $\nqmn$. We now consider the computation of $|\tsri(m,2;q)|$ for $m>1$. By Theorem \ref{tsrisqmn}, it suffices to compute $N_q(m,2)$ which is given by
\begin{align*}
N_q(m,2) &=\left|\left\{X^2-\lambda (aX+1) \in \I{2}{q^m}: \Fqm=\Fq(\lambda),a\in \Fq \right\}\right| \\
 &=\left|\left\{X^2+aX-\alpha \in \I{2}{q^m}: \Fqm=\Fq(\alpha),a\in \Fq \right\}\right|.
\end{align*}
For every positive integer $t>1$ and $a\in\Fq$, define
\begin{equation}
\label{vta}
V_t(a)=\left\{\alpha\in \F_{q^t}:\F_{q^t}=\Fq(\alpha),X^2+aX-\alpha \mbox{ is irreducible in } \F_{q^t}[X]\right\}.
\end{equation}
Then it follows that
$$
N_q(m,2)=\sum_{a\in\Fq}|V_m(a)|.
$$
\begin{proposition}
\label{sdempty}
For $m>1$ and $a\in \Fq$, $V_m(a)=\emptyset$ if and only if $q$ is even and $a=0$. 
\end{proposition}
\begin{proof}
Define
$$
Z_m:=\left\{\alpha\in \Fqm:\Fqm=\Fq(\alpha)\right\}.
$$ 
If $q$ is even and $a=0$ then 
\begin{align*}
V_m(a)  & = \left\{\alpha\in \Fqm:\Fqm=\Fq(\alpha),X^2-\alpha \mbox{ is irreducible in } \Fqm[X]\right\}\\
    & =  \emptyset.
\end{align*}
since every element in $\Fqm$ is a square. Now suppose either $q$ is odd or $a \neq 0$. %  If the polynomial $X^2+aX-\alpha$ is reducible for some $\alpha \in Z_m$, then it has two roots, say $\beta$ and $-a-\beta$ (which are distinct by the assumptions on $q,a$ and $m$), that necessarily lie in $Z_m$. Thus 
% $$
% Z_m \nsubseteq \left\{x^2+ax:x \in Z_m\right\}.
% $$
% Now observe that if $x^2+ax\in Z_m$ for some $x\in \Fqm$ then $x\in Z_m$. Thus 
% $$
% Z_m \nsubseteq \left\{x^2+ax:x \in \Fqm \right\}
% $$
% which implies that $V_m(a)$ is nonempty.  
Consider the map $h:Z_m \to \Fqm$ given by $h(x)=x^2+ax$. For each $\beta \in Z_m$, we have $-a-\beta \in Z_m$ and $h(\beta)=h(-a-\beta)$. Further, $\beta$ and $-a-\beta$ are distinct by the assumptions on $q$, $a$ and $m$. It follows that the range of $h$ is of cardinality $|Z_m|/2$ and thus there exists some $\alpha \in Z_m$ which is not in the range of $h$. Then $X^2+aX-\alpha$ is irreducible in $\Fqm[X]$, implying that $V_m(a)$ is nonempty.
\end{proof}

We will use the above proposition implicitly in the proof of the next theorem which is the main theorem of this paper.

\begin{theorem}
\label{nqm2}
Suppose $m>1$ and $m=2^k l$ where $k,l$ are nonnegative integers with $l$ odd.
\begin{enumerate}
 \item 
If $l=1$, then 
\begin{equation*}
N_q(m,2)=
\begin{dcases}
 \frac{(q-1)q^m}{2} & q \mbox{ even}, \\
\frac{q(q^m-1)}{2}       & q \mbox{ odd}.
\end{dcases}
\end{equation*}
\item If $l>1$, then
\begin{equation*}
N_q(m,2)=
\begin{dcases}
 \frac{l|{\EuScript I}(l;q^{2^k})|}{2} (q-1) & q \mbox{ even}, \\
\frac{l|{\EuScript I}(l;q^{2^k})|}{2} q       & q \mbox{ odd}.
\end{dcases}
\end{equation*}
\end{enumerate}
\end{theorem}

\begin{proof}
For each positive integer $t>1$, let 
$$
Z_t=\left\{\alpha\in \Fqt:\Fqt=\Fq(\alpha)\right\}
$$ 
as in Proposition \ref{sdempty}. Let $a\in\Fq$ and assume that $a\neq 0$ whenever $q$ is even. Define for each positive integer $t>1$ the sets
\begin{align*}
X_t(a) &=\left\{\alpha\in \F_{q^t}:\F_{q^t}=\Fq(\alpha^2+a\alpha)\right\},\\
Y_t(a) &=\left\{\alpha\in \F_{q^t}:\F_{q^t}=\Fq(\alpha)\neq\Fq(\alpha^2+a\alpha)\right\},\\
U_t(a) &=\left\{\alpha^2+a\alpha:\alpha \in \F_{q^t},\F_{q^t}=\Fq(\alpha^2+a\alpha)\right\}. 
\end{align*}
If $V_t(a)$ is as in \eqref{vta}, then it is easy to see that 
\begin{equation}
\label{zisunion}
Z_t=X_t(a)\sqcup Y_t(a)=U_t(a)\sqcup V_t(a).
\end{equation}
Denote the cardinalities of $Z_t,X_t(a),Y_t(a),U_t(a),V_t(a)$ by $z_t,x_t,y_t,u_t,v_t$ respectively. Then by \eqref{zisunion}, it follows that $z_t=x_t+y_t=u_t+v_t$. For each $t>1$, the function $h(x)= x^2+ax$ maps $X_t(a)$ onto $U_t(a)$ and $Y_{2t}(a)$ onto $V_t(a)$. Thus
\begin{equation*}
 \label{recurrence1}
x_t=2u_t \quad \mbox{ and } \quad y_{2t}=2v_t \qquad (t>1).
\end{equation*}
For $0\leq i \leq k$ let $m_i=m/2^i$. Then, for nonnegative $i \leq k-1$ and $m_{i+1}>1$, 
\begin{equation*}
 \label{recurrence2}
x_{m_i}=2u_{m_i} \quad \mbox{ and } \quad y_{m_i}=2v_{m_{i+1}}.
\end{equation*}
If $m$ is odd, then $m \geq 3$ and consequently $y_m=0$ since a field extension of odd degree cannot contain any extension of degree 2. If $m$ is even, then
\begin{align*}
 y_m+x_{m_1}&=2(v_{m_1}+u_{m_1})\\
	    &=2(x_{m_1}+y_{m_1}).
\end{align*}
Thus
\begin{align*}
 y_m&=2(v_{m_1}+u_{m_1})-x_{m_1}\\
	    &=z_{m_1}+y_{m_1}.
\end{align*}
The solution to the recurrence depends on $m$. If $m$ is a power of 2 ($m=2^k$), then 
\begin{align}
y_{m}&=y_{m_{k-1}}+\sum_{i=1}^{k-1}z_{m_i} \nonumber \\
     &=y_2+\sum_{i=1}^{k-1}z_{m_i} \label{power}
\end{align}
where the second summand is understood to be zero when $k=1$. If $m$ is not a power of 2 (i.e. $l>1$ ), then
\begin{align}
y_{m}&=y_{m_k}+\sum_{i=1}^{k}z_{m_i} \nonumber\\
     &=\sum_{i=1}^{k}z_{m_i} \label{nonpower}
\end{align}
since $m_k=l$($\geq 3$) is odd. 

It now remains to compute \eqref{power} and \eqref{nonpower}. First consider \eqref{power} (where $m=2^k$). If $r$ is a power of 2, then $z_r=q^r-q^{r/2}$. A simple calculation shows that
$$
y_m=q^{m/2}-q+y_2.
$$
Now 
\begin{align*}
y_2 & =\left|\left\{\alpha\in \F_{q^2}:\F_{q^2}=\Fq(\alpha)\neq\Fq(\alpha^2+a\alpha)\right\}\right|\\
    &=2 \left|\left\{\alpha \in \Fq:X^2+aX-\alpha \mbox { is irreducible in } \Fq[X]\right\}\right| \\
    &= 2\left(q-|\{s^2+as:s\in\Fq \}|\right) \\
    &= 
\begin{cases}
q-1 & q \mbox{ odd}, \\
q  & q \mbox{ even}.
\end{cases}
\end{align*}
Therefore
\begin{align*}
 |V_m(a)|=v_m &= z_m-u_m \\
        &= \frac{z_m+y_m}{2}\\
    &= \frac{q^m-q+y_2}{2}.
\end{align*}
Thus 
\begin{align*}
 N_q(m,2)=\sum_{a\in\Fq}|V_m(a)|&=
\begin{dcases}
|V_m(1)|(q-1) & q \mbox{ even},\\
|V_m(1)|q & q \mbox{ odd}.
\end{dcases}\\
&=\begin{dcases}
\frac{(q-1)q^m}{2} & q \mbox{ even},\\
\frac{q(q^m-1)}{2} & q \mbox{ odd}.
\end{dcases}
\end{align*}
This settles the first part of the theorem. In the second case ($m=2^kl$, $l>1$), we substitute for $y_m$ from \eqref{nonpower} to obtain
\begin{align*}
v_m=\frac{z_m+y_m}{2}   
    &=\frac{1}{2}\sum_{i=0}^{k}z_{m_i}\\
    &=\frac{1}{2}\sum_{i=0}^{k}z_{2^i l}\\
    &=\frac{z_l}{2}+\frac{1}{2}\sum_{i=1}^{k}z_{2^i l}\\
    &=\frac{z_l}{2}+\frac{1}{2}\sum_{i=1}^{k}\sum_{d\mid 2^i l}\mu(d)q^{\frac{2^i l}{d}}\\
    &=\frac{z_l}{2}+\frac{1}{2}\sum_{i=1}^{k}\sum_{d\mid 2l}\mu(d)q^{\frac{2^i l}{d}} 
\end{align*}    
since $\mu(d)=0$ if $4\mid d$. Since $\mu(2d)=-\mu(d)$ for odd $d$ we can rewrite this as
\begin{align*}
    &\frac{z_l}{2}+\frac{1}{2}\sum_{i=1}^{k}\sum_{d\mid l}\mu(d)\left(q^{\frac{2^i l}{d}}-q^{\frac{2^{i-1} l}{d}}\right) \qquad \\   
    =&\frac{z_l}{2}+\frac{1}{2}\sum_{d\mid l}\mu(d)\left(q^{\frac{2^k l}{d}}-q^{\frac{l}{d}}\right)\\
    =&\frac{z_l}{2}+\left(\frac{1}{2}\sum_{d\mid l}\mu(d)q^{\frac{2^k l}{d}}\right)-\frac{z_l}{2} \\
    =&\frac{l|{\EuScript I}(l;q^{2^k})|}{2}.
\end{align*}
Thus 
\begin{align*}
 N_q(m,2)=\sum_{a\in\Fq}|V_m(a)| &=
\begin{dcases}
\frac{l|{\EuScript I}(l;q^{2^k})|}{2}(q-1) & q \mbox{ even},\\
\frac{l|{\EuScript I}(l;q^{2^k})|}{2}q & q \mbox{ odd}.
\end{dcases}
\end{align*}
This completes the proof of the second part of the theorem.
\end{proof}
\begin{remark} 
\label{nqm}
Suppose $m>1$ and $m=2^k l$ where $k,l$ are integers with $l$ odd. Then Theorem \ref{nqm2} can be stated more compactly as follows:
\begin{align*}
N_q(m,2)&=\left(q-\frac{1+(-1)^q}{2}\right)|V_m(1)|\\
	&=\frac{1}{2}\left(q-\frac{1+(-1)^q}{2}\right)\left(l|{\EuScript I}(l;q^{2^k})|-\left\lfloor\frac{1}{l}\right\rfloor \frac{1+(-1)^{q-1}}{2}\right)
\end{align*}
where $\lfloor x\rfloor$ denotes the floor function. Note that
$$
\left\lfloor\frac{1}{l}\right\rfloor \frac{1+(-1)^{q-1}}{2}=
\begin{cases}
1 &  m \mbox{ is a power of 2 and } q \mbox{ is odd},\\
0 &  \mbox{ otherwise}.
\end{cases}
$$ 
\end{remark}

\begin{theorem}
\label{eq:n=2}
Suppose $m>1$ and $m=2^k l$ where $k,l$ are nonnegative integers with $l$ odd. Then
$$
\left|\tsri(m,2;q)\right|=\left(q-\frac{1+(-1)^q}{2}\right) \left(\sum_{d\mid l}\mu(d)q^{\frac{m}{d}}-\left\lfloor\frac{1}{l}\right\rfloor \frac{1+(-1)^{q-1}}{2}\right) \frac{|\GL_m(\Fq)|}{2m(q^m-1)} .
$$
\end{theorem}
\begin{proof}
Follows from Theorem \ref{tsrisqmn}, Remark \ref{nqm} and the fact that 
$$
|{\EuScript I}(l;q^{2^k})|=\frac{1}{l}\sum_{d\mid l}\mu(d)q^{\frac{2^k l}{d}}.
$$
\end{proof}

\begin{theorem}[Carlitz]
Let $m$ be a positive integer and suppose $m=2^k l$ for some integers $k,l$ with $l$ odd. The number of self-reciprocal irreducible monic (srim) polynomials of degree $2m$ in $\Fq[x]$ is equal to 
$$
\frac{1}{2m}\left(l|{\EuScript I}(l;q^{2^k})|-\left\lfloor\frac{1}{l}\right\rfloor \frac{1+(-1)^{q-1}}{2}\right).
$$
\end{theorem}
\begin{proof}
For $m=1$ we need to count the number of $b$ in $\Fq$ such that $X^2+bX+1$ is irreducible in $\Fq[X]$. The polynomial $X^2+bX+1$ is irreducible precisely when $b$ is not of the form $c+1/c$ for some $c \in \Fq^*$. It is easily seen that
$$
\left|\left\{ c+1/c: c\in \Fq^*\right\}\right|=
\begin{cases}
(q+1)/2 & q \mbox{ odd}, \\
q/2  & q \mbox{ even}.
\end{cases}
$$
The $m=1$ case follows easily from this. Now suppose $m>1$. Let the map $\Gamma:S_q(m,2)\to \Delta_q(m,2)$ be as in Theorem \ref{tsrisqmn}. Since all the fibers of $\Gamma$ are of size $m$, it follows that 
% If we set $g(X)=1+X$  and $n=2$ in \eqref{decomp} it is clear from the proof of Theorem \ref{tsrisqmn} that 
the number of polynomials in $\Delta_q(m,2)$ of the form $(1+X)^m h\left(\frac{X^2}{1+X}\right)$ is equal to
$$
\frac{1}{m}\left|\left\{X^2-\lambda (X+1) \in \I{2}{q^m}: \Fqm=\Fq(\lambda) \right\}\right|=\frac{|V_m(1)|}{m}.
$$
Now
% \begin{align*}
%  &\left|{\EuScript I}(2m;q)\cap \left\{(1+X)^m h\left(\frac{X^2}{1+X}\right): h(X)\in \Fq[X], \deg h(X)=m \right\} \right|\\
%  =&\left|{\EuScript I}(2m;q)\cap \left\{(X^m h\left(\frac{(X-1)^2}{X}\right): h(X)\in \Fq[X], \deg h(X)=m \right\} \right|\\
%  =&\left|{\EuScript I}(2m;q)\cap \left\{(X^m h\left(X+\frac{1}{X}-2\right): h(X)\in \Fq[X], \deg h(X)=m \right\} \right| \\
%  =&\left|{\EuScript I}(2m;q)\cap \left\{(X^m h\left(X+\frac{1}{X}\right): h(X)\in \Fq[X], \deg h(X)=m \right\} \right| 
%  \end{align*}
\begin{align*}
(1+X)^m h\left(\frac{X^2}{1+X}\right) \mbox{ is irreducible} &\Leftrightarrow X^m h\left(\frac{(X-1)^2}{X}\right) \mbox{ is irreducible}  \label{eq:srim}\\
&\Leftrightarrow X^m h_1\left(X+\frac{1}{X}\right) \mbox{ is irreducible} \notag
\end{align*}
where $h_1(X)=h(X-2)$. Irreducible polynomials of the form $X^m h_1\left(X+\frac{1}{X}\right)$ where $h_1$ is monic of degree $m$ are precisely the srim polynomials of degree $2m$. Thus the number of such polynomials is equal to $|V_m(1)|/m$. This is the statement of the theorem.
\end{proof}

\begin{remark}
Any irreducible self-reciprocal polynomial of degree $\geq 2$ over $\Fq$ is necessarily of even degree.
\end{remark}

\begin{corollary}
For every positive integer $m$, $f(X)\in \Delta_2(m,2)$ if and only if $f(X-1)$ is a srim polynomial of degree $2m$. 
\end{corollary}
\begin{proof}
This follows easily since  polynomials in $\Delta_2(m,2)$ are  precisely the irreducible polynomials of the form
$$
(1+X)^m h\left(\frac{X^2}{1+X}\right)
$$
where $h$ is monic of degree $m$.
\end{proof}

% \begin{corollary}
% If $f(X)$ is a srim polynomial of degree $2m$ over $\Fq$, then one can easily construct a matrix in $\tsri(m,2;q)$ with $f(X+1)$ as its characteristic polynomial.
% \end{corollary}

\begin{remark}
If $f(X)$ is a srim polynomial of degree $2m$ over $\Fq$, then $f(X+1)$ is the characteristic polynomial of some matrix in $\tsri(m,2;q)$. Thus we can easily construct matrices in $\tsri(m,2;q)$ from srim polynomials.
\end{remark}

\section{Bounds on the Number of Irreducible TSRs}
\begin{theorem}
\begin{align*}
 |\tsrimnq| & \leq \frac{|\GL_m(\Fq)|}{q^m-1}|{\EuScript I}(m;q)|q^{n-1}.\\
 |\tsrpmnq| & \leq \frac{|\GL_m(\Fq)|}{q^m-1}|{\EuScript P}(m;q)|q^{n-1}.
 \end{align*}
\end{theorem}
\begin{proof}
First note that $T$ is uniquely determined by $g_T(X)$ and $T_{(m)}$ (as in \eqref{charTSR}). If $T\in \tsrimnq$, then $\phitm(X)$ is irreducible of degree $m$ and there are at most $q^{n-1}$ possibilities for $g_T(X)$. The first bound easily follows from these observations. The second bound follows similarly by using Proposition~\ref{primblock}.
\end{proof}
%replace2%
\bibliographystyle{abbrv}
\bibliography{/home/samrith/Dropbox/math/bibliography/mybib}

\end{document}